\documentclass{amsart}
\usepackage{amsmath,amsfonts,amsthm,amssymb}
\usepackage{cite}
\usepackage[utf8]{inputenc}
\usepackage{amsmath}
\usepackage{amsfonts}
\usepackage{amssymb}
\usepackage{amsthm}
\usepackage{mathalfa}
\usepackage{graphicx}
\usepackage[colorinlistoftodos]{todonotes}
\usepackage{enumitem}
\usepackage{xfrac}
\usepackage[margin=1in]{geometry}
\usepackage{mathtools}
\usepackage[all]{xy}
\usepackage{bm}

\newif\ifthesis
\thesistrue

\usepackage[colorlinks=false]{hyperref}

\newcommand{\CC}{\mathbb{C}}

\newcommand{\ZZ}{\mathbb{Z}}
\newcommand{\QQ}{\mathbb{Q}}
\newcommand{\NN}{\mathbb{N}}

\newcommand{\lam}{\lambda}

\newcommand{\FF}{\mathbb{F}}

\newcommand{\then}{\Rightarrow}

\newcommand{\calD}{\mathcal{D}}

\newcommand{\alp}{\alpha}

\usepackage{aliascnt}

\theoremstyle{definition}
\newtheorem{theorem}{Theorem}[section]

\newaliascnt{lemct}{theorem}

\newtheorem{lem}[lemct]{Lemma}
\aliascntresetthe{lemct}

\newaliascnt{propct}{theorem}
\newtheorem{prop}[propct]{Proposition}
\aliascntresetthe{propct}

\newaliascnt{corct}{theorem}
\newtheorem{cor}[corct]{Corollary}
\aliascntresetthe{corct}

\newaliascnt{exmct}{theorem}

\aliascntresetthe{exmct}

\newaliascnt{exerct}{theorem}

\aliascntresetthe{exerct}

\newaliascnt{discct}{theorem}
\newtheorem{disc}[discct]{Discussion}
\aliascntresetthe{discct}

\newaliascnt{rmkct}{theorem}
\newtheorem{rmk}[rmkct]{Remark}
\aliascntresetthe{rmkct}

\newaliascnt{clmct}{theorem}

\aliascntresetthe{clmct}

\newaliascnt{qstct}{theorem}

\aliascntresetthe{qstct}

\newaliascnt{deffct}{theorem}
\newtheorem{deff}[deffct]{Definition}
\aliascntresetthe{deffct}

\usepackage{thmtools}
\usepackage{thm-restate}

\newcommand{\fg}{finitely generated }

\newcommand{\st}{such that }

\newcommand{\twlogd}[1]{without loss of generality{#1}}
\newcommand{\twlog}{\twlogd{ }}

\newcommand{\tiff}{if and only if }

\newcommand{\wrt}{with respect to }

\usepackage{marginnote}

\newcommand{\mth}{^{\text{th}}}

\newcommand{\bref}[1]{\textbf{\autoref{#1}}}
\newcommand{\pref}[1]{(\ref{#1})}

\def\O{{\mathcal O}}

\newcommand{\mchar}{\mathrm{char}}

\newcommand{\sbs}{\subset}

\def\P{{\mathbb P}}
\newcommand{\PP}{\mathbb P}

\newif\ifdebug
\debugfalse

\usepackage{verbatim}
\usepackage{courier}

\newcommand{\labelt}[1]{\label{#1}\ifdebug \texttt{\smaller{#1     }} \fi}

\numberwithin{equation}{theorem}

\title{ An Elementary computation of the $F$-Pure Threshold of an Elliptic Curve}  

\author{Gilad Pagi}
\thanks{The author  acknowledges the financial support of NSF grant DMS-0943832.}

\begin{document}

\maketitle
\begin{abstract}
We compute the $F$-pure threshold of a degree three homogeneous polynomial in three variables with an isolated singularity. The computation uses elementary methods to prove a known result of Bhatt and Singh (from \cite{BBelliptic}).
\end{abstract}
\section{Introduction}

In this note, we provide an alternative and elementary proof for a known result about the $F$-pure threshold of a  homogeneous polynomial of degree three in three variables with an isolated singularity. Such a polynomial defines an elliptic curve in $\PP^2$. Let $K$ denote a field of prime characteristic $p$ and let $R=K[x_1,...,x_t]$. Fix any polynomial $f\in R$. By $F$-pure threshold we mean:
 \begin{equation}\labelt{FTdef}
 FT(f):=\sup\left\{ \frac{N}{p^e}\mid  N,e\in  \ZZ_{>0},  f^N\not\in (x_1^{p^e},...,x_t^{p^e})R   \right\},
 \end{equation} a definition  that first appeared in \cite{BMS09}, although the first formulation using tight closure theory is stated in \cite{HW02}. 
 
 The $F$-pure threshold is a numerical measurement of the singularity of $f$ at the origin. If $f$ is smooth there, $FT(f)=1$. Smaller values of $FT(f)$ mean ``worse singularities" of $f$ at the origin. The $F$-pure threshold is a characteristic $p$ analog of the log canonical threshold of a complex singularity (see \cite{KOL97}). When $f$ is defined over $\CC$, one can reduce to the characteristic $p$ case, compute $FT(f)$ and compare the values in different primes $p$ to the log canonical threshold. The limit of $FT(f)$ when $p\to \infty$ approaches the log canonical threshold of $f$ \cite[Theorems 3.3,3.4]{MTW05}. This fact is the culmination of a series of papers, going back to \cite{HH90}, \cite{Smi00}, \cite{H01}, \cite{HW02}, \cite{HY03}, \cite{T04}, \cite{HT04}, \cite{TW04}. See the survey \cite{KSbasic} for a gentle introduction. 
 
The $F$-pure threshold of the defining equation of an elliptic curve in $\P^2$ is closely related to supersigularity. Recall the definition of supersingularity of an elliptic curve $E$ in characteristic $p>2$. The Frobenius morphism $E \overset{F} \longrightarrow E$ induces a map $H^1(E,\O_E) \overset{F^*} \longrightarrow H^1(E,\O_E)$. Then $E$ is defined to be supersingular if $F^*$ is the zero map. Otherwise, $E$ is ordinary. 

For our purpose, we adopt a more concrete characterization of supersingularity, in terms of the Hasse invariant of the defining polynomial $f$ of $E$ in $\PP^2$. We review and develop this point of view in \bref{ssDef}. See also \cite[IV.4]{HARTS} and \cite[V.3,V.4]{SILV}.

In the upcoming sections we present an elementary proof of the following result of Bhatt and Singh:
\begin{theorem}[\textbf{Main Theorem}]\labelt{mainThm}
 Let $K$ denote a field of prime characteristic $p>2$. Let $f\in K[x,y,z]$ be a homogeneous polynomial of degree three defining an elliptic curve $E$ in $ \P_K^2$.
 Then:
 \[ FT(f)=\left\{
\begin{array}{ll}
1& \text{ if } E \text{ is ordinary}\\
1-\frac{1}{p}& \text{ if } E \text{ is supersingular}
\end{array}\right.
\]
\end{theorem}

Bhatt and Singh provide a couple of proofs in \cite{BBelliptic} using a translation into local cohomology;
Generalizations can be found in \cite{HNWZ16}. In contrast, our approach involves directly investigating the form of $f$ raised to integer powers using a generalized formula of the well known polynomial $H_p(\lam)=\sum_{i=0}^n {m\choose i}^2\lam^i$ with $m=(p-1)/2$, used to compute the Hasse invariant. See also \cite{MUL1},\cite{MUL2}. 

Going back to the characteristic zero case, for infinitely many primes $p$, the reduction of an elliptic curve mod $p$ is ordinary (e.g. over $\QQ$,  see\cite[Exercise 5.11]{SILV}). So we see that not only the $F$-pure threshold approaches the log canonical threshold, but it actually equals the log canonical threshold for infinitely many primes. For a general polynomial, this remains an open question (see some progress \cite{HerLog})

\subsection*{Acknowledgments}
This article is part of my Ph.D. thesis, which is being written under the direction of Karen Smith of University of Michigan. I would like to thank Prof. Smith for many useful discussions. Many thanks to Prof. Daniel Hern\'{a}ndez for his remarks on the earlier draft and to Prof. Michael Zieve, Prof. Sergey Fomin and Prof. Bhargav Bhatt for fruitful conversations. Finally, I would like to thank the referee of the Journal of Algebra for their helpful comments.

\section{Discussion}
Let $K$ denote a field of prime characteristic $p>2$. Let $f\in K[x,y,z]$ be homogeneous polynomial of degree three with an isolated singularity. Let $E\sbs \P^2$ be the elliptic curve defined by $f$. Note that the supersingularity of $E$ and the value of $FT(f)$ are invariant under passing to the algebraic closure $\overline{K}$, and under change of coordinates. So \twlog we assume $K$ is algebraically closed and change coordinates so $f$ is in its Legendre form:
$$
f_a(x,y,z) = y^2z - x(x-z)(x-a z),\,\,\, a \in K - \{0,1\}
$$
By letting $a$ range over $K - \{0,1\}$ we are addressing all possible elliptic curves in $\P^2$ up to isomorphism. Thus, it suffices to prove the \nameref{mainThm} for this one-parameter family of polynomials.

 Working with $f_a$ allows us to assert supersingularity by a simple computation on $a$.  We are going to work with the following, as proven in \cite[IV, Corollary 4.22]{HARTS}. 
\begin{prop}\labelt{ssDef} Let $K$ be a field of prime characteristic $p>2$.
 Let $f_a(x,y,z) = y^2z - x(x-z)(x-a z) \in K[x,y,z]$, with $a \in K-\{0,1\}$. Let $E\sbs \P^2$ be the elliptic curve defined by $f_a$. Then $E$ is supersingular \tiff over $K$:
 $$
\sum_{i=0}^m {m\choose i}^2a^i = 0, \text{ with }m=(p-1)/2,
 $$
 that is, \tiff $a$ is a root of the polynomial
 $$
 H_p(\lam) = \sum_{i=0}^m {m\choose i}^2\lam^i,\text{ with }m=(p-1)/2
 $$ in $K[\lam]$.
 Otherwise, $E$ is ordinary. 
\end{prop}
\noindent In particular, if $a$ is transcendental over $\overline{\FF_p}$, the polynomial $f_a\in K[x,y,z]$ always defines an ordinary elliptic curve.\\ 

It turns out that when investigating integer powers of $f_a$, one gets coefficients similar to the form of $H_p(a)$, as we prove later in the \nameref{hassePolyGeneral}. This motivates the following definition:
 \begin{deff}\labelt{ghp} Let $n \in \ZZ_{\geq 0}$. Define the following polynomial in $\ZZ[\lam]$:
 $$
 H\left\{n\right\}(\lam): = \sum_{i=0}^n {n\choose i}^2\lam^i
 $$
 \end{deff}
 Following \cite{MOR}, we call it the \textit{Deuring Polynomial}\footnote{Arguably it first appeared in \cite{DUR}} of degree $n$. When the indeterminant $\lam$ is understood from the context we omit it and write $H\{n\}$. We often abuse notation and write $H\{n\}\in \FF_p[\lam]$ for the natural image mod $p$. For an odd prime $p$, the polynomial $H\left\{\frac{p-1}{2}\right\}$ is $H_p(\lam)$ and plays an important role in number theory, as we saw in \bref{ssDef}. We shall dedicate the next section to investigate the connection of $H\{n\}$ to our problem and prove interesting properties of it.
 
 To make notation more compact, for a fixed $p$ and a nonnegative integer $e$ we define:
 \begin{equation}\labelt{noneeq}
\begin{array}{llll}
&N_e &=& p^e-1\\
n_e = &N_e/2 &=& \frac{p^e-1}{2},
\end{array}
 \end{equation}
Specifically, when $e=1$ we have:
$$
n_1 = \frac{p-1}{2}.
$$
Using \bref{ssDef} we can rewrite the \nameref{mainThm} in a more computationally-friendly version:

\begin{theorem}[\textbf{Main Theorem V2}]\labelt{mainThmv2}
 Let $K$ denote a field of prime characteristic $p>2$. Let $f_a(x,y,z) = y^2z - x(x-z)(x-a z) \in K[x,y,z]$, with $a \in K-\{0,1\}$. Let $n_1=(p-1)/2$.
 Then:
 \[ FT(f_a)=\left\{
\begin{array}{ll}
1& \text{ if } H\{n_1\}(a)\not\equiv 0 \pmod p\\
1-\frac{1}{p}& \text{ if } H\{n_1\}(a)\equiv 0 \pmod p
\end{array}\right.
\]
\end{theorem}
\noindent When $H\{n_1\}(a)\not\equiv 0 \pmod p$, we say that $f_a$ is ordinary. Otherwise we say that $f_a$ is supersingular. \\

The next section is dedicated to develop the required machinery. Afterwards we prove  \nameref{mainThmv2} directly.

\begin{rmk} The Deuring polynomials $H\{m\}$ are closely related to the Legendre polynomials araising as solutions to the Legendre differential equation. Legendre polynomials are of importance to many physical problems, including finding the gravitational potential of a point mass, as in Legendre's original work \cite{LEG1785}. Indeed, If $P_m(x)$ denotes the $m\mth$ Legendre polynomial then:
$$
H\{m\}(\lam)=(1-\lam)^mP_m\left(\frac{1+\lambda}{1-\lambda}\right), 
$$
as follows by a simple substitution and a known ``textbook" formula for the Legendre polynomials (\cite[Exercise 2.12]{KOEPF}); this is pointed out in \cite{CULL14} and \cite{BrillMor04}. In section 3, we establish several properties of Deuring polynomials, which can also be deduced from analogous facts about Legendre polynomials. We include direct algebraic proofs not relying on typical analytic techniques such as orthogonality in function spaces. In this way, we keep our paper self-contained and, we hope, more straightforward than relying on the vast literature on Legendre polynomials. 
\end{rmk}

\section{Deuring Polynomials and Machinery}

We first recall some well known techniques for working in characteristics $p$. Fix a prime $p$. Every integer $N$ can be written \textit{uniquely} in its \textit{base $p$-expansion} ($p$-expansion for short) as follows: fix a power $e$ \st $N<p^{e+1}$. Then there exist unique integers $0\leq a_0,...,a_{e}\leq p-1$ \st:
$$
N = a_{0}p^0 + a_{1}p^1 + ... + a_{e}p^e
$$
We recall how to compute binomial and multinomial coefficients mod $p$. 
\begin{theorem}[\textbf{Lucas's Theorem}]\labelt{luc} [See \cite{LUC} and \cite{DIC}] Let $\textbf{k}=(k_1,...,k_n)\in \NN^n$ and set $N=k_1+...+k_n$. Fix a prime $p$. Let $e$ be an integer \st $N<p^{e+1}$. Write each of the $k_i$ in its base $p$-expansion:
$$
k_i = a_{i0}p^0 + a_{i1}p^1 + ... + a_{ie}p^e
$$
(some $a_{ij}$'s may be 0). Also write $N$ in its base $p$-expansion:
$$
N = b_{0}p^0 + b_{1}p^1 + ... + b_{e}p^e
$$
 Then the multinomial coefficient ${N \choose \textbf{k}}$ satisfies:
$$
{N \choose \textbf{k}} = \frac{N!}{k_1!\cdots k_n!} \equiv {b_0 \choose a_{10}\,a_{20}\,...\,a_{n0}}{b_1 \choose a_{11}\,a_{21}\,...\,a_{n1}}\cdots {b_e \choose a_{1e}\,a_{2e}\,...\,a_{ne}} \pmod p,
$$
 with the convention that if $a_{1j}+...+a_{nj}>b_j$ then ${b_j \choose a_{1j}\,a_{2j}\,...\,a_{nj}}=0$. Specifically, ${N \choose \textbf{k}}  \not \equiv 0 \pmod p$ \tiff the digits of the $p$-expansion of the $k_i$'s do not carry when added (if $a_{ij}$ is the $j\mth$ digit of $k_i$, the last condition amounts to: for all $0\leq j\leq e$,  $0\leq a_{1j} + ... + a_{nj}=b_j<p$).
\end{theorem}
Due to  \nameref{luc}, a multinomial coefficient is 0 \tiff for some $j$, the $j\mth$ digit of $N$ is not the sum of the of the $j\mth$ digits of the $k_i$'s.

\bigskip
The next lemma shows that understanding the Deuring polynomial $H\{n\}$ of \bref{ghp} is crucial for the discussion.

\begin{lem}[\textbf{Main Technical Lemma}]\labelt{hassePolyGeneral}
 Let $f_\lam=y^2z - x(x-z)(x-\lam z)$ and let $N=n+m$ be an integer. Then the coefficient of $x^{2m}y^{2n}z^{n+m}$ in $f_\lam^N$ is ${N \choose n} H\{m\}(\lam)$, up to sign.
\end{lem}
\begin{proof}
Observe $(y^2z - x(x-z)(x-\lam z))^{n+m}$. Since $y$ is only in the left term, we need to raise it to the power of $n$. This gives the binomial coefficient ${n+m \choose n}$. So it is left to identify the coefficient of $x^{2m}z^{m}$ in $(- x(x-z)(x-\lam z))^{m}=(-1)^{m}x^{m}(x-z)^m(x-\lam z)^{m}$. The latter allows us to just compute the coefficient of $x^{m}z^{m}$ in $(x-z)^{m}(x-\lam z)^{m}$. Notice:
$$
(x-z)^{m}(x-\lam z)^{m} = \left(\sum_{i=0}^{m} {m \choose i} (-1)^{m-i}x^{i}z^{m-i}\right) \left(\sum_{j=0}^{m} {m\choose j} (-\lam)^{j}x^{m-j}z^{j}\right).
$$
For the coefficient of $x^{m}z^{m}$ we need to set $i=j$, so we end up with:
$$
(-1)^{m} \sum_{i=0}^{m} {m \choose i}^2 \lam^{i} = (-1)^mH\{m\}.
$$
Together, up to sign, we get ${n+m \choose n}H\{m\}$. 
\end{proof}

\begin{cor}\labelt{hassePoly}
 Let $f_\lam=y^2z - x(x-z)(x-\lam z)$ and let $N=2n$. So the coefficient of $x^{2n}y^{2n}z^{2n}$ in $f^N$ is ${2n \choose n} H\{n\}(\lam)$ up to sign.
\end{cor}
\begin{proof}
Apply the \nameref{hassePolyGeneral} with $m=n$. 
\end{proof}

The \nameref{hassePolyGeneral} motivates us to investigate the roots of $H\{n\}$ in characteristics $p$.

\begin{lem}\labelt{pMinusOneHass}
Let $p$ be a prime. Then $H\{p-1\}\in \FF_p[\lam]$ is $(\lam-1)^{p-1}$.
\end{lem}
\begin{proof}
The coefficients of $H\{p-1\}(\lam)$ are the squares of the numbers appearing on the $(p-1)\mth$ row in Pascal's Triangle mod $p$. Due to \nameref{luc}, the $p\mth$ row starts and ends with $1$, while the rest of the entries are zero. Ergo, the $(p-1)\mth$ row consists of $\pm 1$'s due to the identity:
\begin{equation}\labelt{pascalIden}
{n-1 \choose i-1} + {n-1 \choose i} = {n \choose i}.
\end{equation}
For illustration, here are the $(p-1)\mth$ and the $p\mth$ rows of Pascal's Triangle:
$$
\begin{array}{rcccccccccccccccccccccccc}
p-1:&&&1&&-1&&1&&-1&&...&&-1&&1&&-1&&1\\
p:&&1&&0&&0&&0&&...&&...&&0&&0&&0&&1
\end{array}
$$
So using the geometric series formula we get: 
$$
H\{p-1\}=1+\lam+...+\lam^{p-1} = \frac{\lam^p-1}{\lam-1} = (\lam-1)^{p-1}
$$
\end{proof}

\begin{lem}\labelt{hasseLucas}\footnote{This lemma was formulated by Schur in the context of Legendre polynomials, which are closely related to the Deuring polynomials. However, the first published proof is due to Wahab(\cite{WH52}) half a decade later. We provide a simple proof in the context of Deuring polynomials.} Fix a prime $p$. 
Let $H\{n\}\in \FF_p[\lam]$. Write the $p$-expansion of $n$:
$$
n=b_0p^0+b_1p^1+...+b_ep^e.
$$

Then
$$
H\{n\} = H\{b_0\}^{1}H\{b_1\}^{p^1}H\{b_2\}^{p^2}\cdots H\{b_e\}^{p^e}
$$
\end{lem}
\begin{proof}
Denote $f=H\{n\}$ and $g=H\{b_0\}^{1}H\{b_1\}^{p^1}\cdots H\{b_e\}^{p^e}$. First notice that $f$ and $g$ are of the same degree as $\deg f = n$ and $\deg g = b_0 + b_1p + b_2p^2 + ... + b_ep^e=n$. Fix $\lam^i$ and let us compare its coefficient in both $f$ and $g$. For $i=0$, the coefficient of $\lam^0$ is 1 in any Deuring polynomial, and so in $f$ and in $g$. Now fix $0<i\leq n$. In $f$, the coefficient is
$$
{ n \choose i}^2.
$$
To compute the coefficient in $g$, write $i$ in its base $p$-expansion:
$$
i=a_0p^0+a_1p^1+...+a_ep^e,
$$
so
$$
\lam^i=\lam^{a_0p^0}\lam^{a_1p^1}\cdots \lam^{a_ep^e}.
$$

Note that the largest power $e$, as appears in the expansion of $n$, is sufficient as $i\leq n$. Notice that the powers of $\lam$ in $H\{b_j\}^{p^j}$ can only be $\{0p^j,1p^j, 2p^j, ..., b_jp^j\}$. So if $j_1\neq j_2$ then the set of powers in $H\{b_{j_1}\}^{p^{j_1}}$ and in $H\{b_{j_2}\}^{p^{j_2}}$ are disjoint except for $0$. It is easy to see that by picking one monomial in each of factors of $g$ and multiplying them together, one gets a monomial $\lam^{\ell}$ in $g$ where the different chosen factors ``spell out" the $p$-expansion of $\ell$. Due to uniqueness of the $p$-expansion of $i$, there is only one possible combination of terms in the different $H\{b_j\}^{p^j}$s that can yield the monomial $$
\lam^i=\lam^{a_0p^0}\lam^{a_1p^1}\cdots \lam^{a_ep^e}.
$$ Namely, we need to follow the $p$-expansion of $i$ and choose $\lam^{a_0}$ from $H\{b_0\}(\lam)^{p^0}$, $\lam^{a_1p}$ from $H\{b_1\}(\lam)^{p^1}$ and so on. 
\[
\begin{array}{ccccccccc}
g = &H\{b_0\}^{1}&H\{b_1\}^{p^1}&H\{b_2\}^{p^2}&\cdots &H\{b_e\}^{p^e}\\
\lam^i = & \lam^{a_0p^0}&\lam^{a_1p^1}&\lam^{a_2p^2}&\cdots& \lam^{a_ep^e}
\end{array}
\]

Ergo, if $a_j\leq b_j$ for all $1\leq j \leq e$, then $\lam^i$ appears in $g$ with a coefficient of:
$$
{b_0 \choose a_0}^{2}{b_1 \choose a_1}^{2p}\cdots {b_e \choose a_e}^{2p^e}.
$$
By Fermat's little theorem, the expression is:
$$
{b_0 \choose a_0}^{2}{b_1 \choose a_1}^{2}\cdots {b_e \choose a_e}^{2},
$$ which is precisely the coefficient of $\lam^i$ in $f$ due to \nameref{luc}.
Otherwise, if for some $j$, $a_j>b_j$, then $\lam^i$ is not in $g$, and its coefficient in $f$ is 0 as well since $i$ and $n-i$ are carrying in the $j\mth$ digit when added and thus ${ n \choose i}=0$. 
\end{proof}

\begin{cor}\labelt{neHasse} In characteristic $p$:
$$
H\left\{\frac{p^e-1}{2}\right\} = H\left\{\frac{p-1}{2}\right\}^{1+p+...+p^{e-1}}
$$
\end{cor}
\begin{proof}
We apply \bref{hasseLucas} after writing $\frac{p^e-1}{2}$ in its $p$-expansion and using geometric series formula:
$$
\frac{p^e-1}{2} = \frac{p-1}{2}(1+p+...+p^{e-1}) =  \frac{p-1}{2} + \frac{p-1}{2}p + ...+\frac{p-1}{2}p^{e-1}
$$
\end{proof}

Recall that we denote $n_e=(p^e-1)/2$ and then $n_1=(p-1)/2$. We can rewrite \bref{neHasse} as
$$
H\left\{n_e\right\} = \left(H\left\{n_1\right\}\right)^{1+p+...+p^{e-1}}
$$
Note that $H\left\{n_1\right\}$ is the polynomial appearing in \bref{ssDef}, so it has an important role in the context of our \nameref{mainThm}. 

\bigskip
In our proof of the \nameref{mainThmv2} we will encounter another polynomial: $H\{n_1-1\}$. We shall now investigate it.

\begin{lem}\labelt{pascalConnection} Fix an integer $n$.
Let $F(\lam)\in \QQ[\lam]$ the formal antiderivative of the polynomial $H\{n-1\}(\lam)$ with constant coefficient $0$. We denote $H\{n-1\} = F'$. Then
$$
 (1-\lam)F' + 2nF = H\{n\}.
$$ Note that this equality holds characteristic zero and thus in all positive characteristics $p > n$.
\end{lem}
\begin{proof}
Let us give a specific formula for $F(\lam)$:
$$
F(\lam)  = \sum_{i=0}^{n-1} {n-1 \choose i}^2 (i+1)^{-1} \lam^{i+1} = \sum_{i=1}^{n}{n-1 \choose i-1}^2 (i)^{-1} \lam^{i}.
$$ 
Now, observe:
$$
(1-\lam)H\{n-1\} + 2nF = \sum_{i=0}^{n-1} {n-1 \choose i}^2\lam^i - \sum_{i=0}^{n-1} {n-1 \choose i}^2\lam^{i+1} + 2n \sum_{i=1}^{n}{n-1 \choose i-1}^2 (i)^{-1} \lam^{i}.$$
Shift the index of the middle sum to get:

\begin{equation}\labelt{alltems}
=\sum_{i=0}^{n-1} {n-1 \choose i}^2\lam^i - \sum_{i=1}^{n} {n-1 \choose i-1}^2\lam^{i} +  \sum_{i=1}^{n}2{n-1 \choose i-1}^2 \frac{n}{i} \lam^{i}.
\end{equation}

For $i=0$, we get that only the leftmost sum contributes a constant coefficient, which is 1 as required. Now consider the case where $1\leq i \leq n$. We need the following identity to simplify the rightmost sum: 
$$
2{n-1 \choose i-1}^2 \frac{n}{i} = 2{n-1 \choose i-1}^2 \frac{n-i+i}{i}=2{n-1 \choose i-1}^2 \left(\frac{n-i}{i}+1\right) = 2{n-1 \choose i-1}{n-1 \choose i} + 2{n-1 \choose i-1}^2.
$$ So when $i$ is fixed, the coefficient of $\lam^i$ in (\ref{alltems}) is
$$
{n-1 \choose i}^2  - {n-1 \choose i-1}^2 + 2{n-1 \choose i-1}{n-1 \choose i} + 2{n-1 \choose i-1}^2
$$
Combining like terms simplifies as:
$$
{n-1 \choose i-1}^2 + 2{n-1 \choose i-1}{n-1 \choose i}  + {n-1 \choose i}^2, $$
which further simplifies as:
$$=\left({n-1 \choose i-1}+{n-1 \choose i}\right)^2 ={n \choose i}^2
$$
using the known identity (\ref{pascalIden}).
So we conclude:
$$
 (1-\lam)H\{n-1\} + 2nF = H\{n\}.
$$
\end{proof}

\begin{lem}\labelt{diffOp} Fix a prime $p>2$. Recall $n_1=(p-1)/2$. Then the following holds over any field $K$ in characteristic $p$:
\begin{enumerate}
\item Let $F(\lam)\in K[\lam]$ be the formal antiderivative of the polynomial $H\{n_1-1\}(\lam)$ with constant coefficient $0$. Then $F$ has no repeated roots. 
\item $H\{n_1\}\in K[\lam]$ has no repeated roots. Further, $\lam=0,1$ are not roots of $H\{n_1\}$.
\end{enumerate}
\end{lem}
\begin{proof}\mbox{}
\begin{enumerate}

\item 
We will show that over $K$ $F=F(\lam)$ satisfies the following differential equation:
\begin{equation} \labelt{eqHnemo}
4\lam(\lam-1)F'' + 8\lam F' + F = 0,
\end{equation}
where $F',F''$ are the first and second derivatives \wrt $\lam$, respectively.  
Once we prove (\ref{eqHnemo}) we see that the only possible repeated roots of $F$ can be 0 or 1 by the following argument: Suppose $\alp$ is a root of $F$ of multiplicity $r\geq 2$. Since $\deg F = n_1 = (p-1)/2$, then $r<p$. So write
\[
\begin{array}{llcl}
F = &g_1(\lam) \cdot (\lam-\alp)^r&\text{ where }&g_1(\alp) \neq 0 ,\\
F' =&g_2(\lam) \cdot  (\lam-\alp)^{r-1}&\text{ where }&g_2(\alp)\neq 0 ,\\
F'' =&g_3(\lam) \cdot  (\lam-\alp)^{r-2}&\text{ where }&g_3(\alp) \neq 0 .\\
\end{array}
\]
Plug the above expression in (\ref{eqHnemo}) and divide by $(\lam-\alp)^{r-2}$ to get
$$
4\lam(\lam-1)g_3 + 8\lam (\lam-\alp)g_2 + (\lam-\alp)^2g_1 = 0 .
$$Plugging in $\lam=\alp$ gives:
$$
4\alp(\alp-1)g_3(\alp)= 0
$$ Since $p\neq 2$, 4 is a unit. We get:
$$
\alp(\alp-1)\equiv 0 \pmod p \then \alp=0,1 
$$i.e. the only possible repeated roots of $F$ are $\alp=0$ or $\alp=1$.

 While 0 is a root of $F$, it is simple since $F'(0)=H\{n_1-1\}(0)=1$. In addition, $\lam=1$ is not a root of $F'(\lam)$ as the following combinatorial identity (which holds over $\ZZ$) shows:
$$
F'(1)= H\{n_1-1\}(1)=\sum_0 ^{n_1-1} {n_1-1 \choose i}^2 = {2n_1-2 \choose n_1-1} = {p-3 \choose n_1-1},
$$
which is not zero in $K$ by \nameref{luc}.

All that is left to do is to show that the differential equation (\ref{eqHnemo}) holds. This can be done by checking the coefficient of $\lam^i$, $0\leq i \leq n_1$ in the different summands. Note that we are working over a field of characteristic $p$ so $2n_1=p-1=-1$. Also recall that we are using the convention that if $k<0$ then ${n\choose k}=0$:
\[
\begin{array}{llll}
\text{coefficient in } F:& {n_1-1 \choose i-1}^2 \frac{1}{i}\\
\text{coefficient in } F'=H\{n_1-1\}:& {n_1-1 \choose i}^2 \\
\text{coefficient in } 8\lam F':& {n_1-1 \choose i-1}^2 8\\
\text{coefficient in } 4\lam F'':& 4{n_1-1 \choose i}^2 (i) &=& 
{n_1-1 \choose i-1}^2 4\frac{(n_1-1-(i-1))^2(i)}{(i)^2} =  \\
&&=&{n_1-1 \choose i-1}^2 \frac{(2n_1-2i)^2}{(i)}
= {n_1-1 \choose i-1}^2 \frac{(-1-2i)^2}{(i)} \\
\text{coefficient in } 4\lam^2F'':& {n_1-1 \choose i-1}^2 4(i-1)
\end{array}
\]
Notice that for $i=0$, we have $i-1<0$ so all the coefficients are 0. Now compute the coefficient of $\lam^i$ with $i>0$ in 
$$
4\lam^2F'' -4\lam F'+8\lam F' + F.
$$
We get:
$$
\frac{{n_1-1 \choose i-1}^2}{i}\left(4(i-1)i - (-1-2i)^2 + 8i + 1\right) = \frac{{n_1-1 \choose i-1}^2}{i}\left(4i^2-4i-1-4i-4i^2+8i+1\right)=0.
$$
\item This is proved in \cite{IGUS58} (see also \cite[Theorem 4.1]{SILV}), and we provide a sketch. First we show that $H\{n_1\}$ satisfies a differential operator similar to \pref{eqHnemo} (which is called the \textit{Picard-Fuchs} operator):
$$
\calD_{PF} =  4\lam(1-\lam)\frac{d^2}{d\lam^2} + 4(1-2\lam) \frac{d}{d\lam} - 1
$$
One can get $\calD_{PF}$ by using an argument similar to the one above or  by simply taking a derivative of \pref{eqHnemo}. We then deduce that the only possible repeated roots are $\lam=0,1$. But since $H\{n_1\}(0)\neq 0$ and $H\{n_1\}(1)\neq 0 $ (can be computed directly), $H\{n_1\}(\lam)$ has no repeated roots in over $K$.
\end{enumerate}

\end{proof}
\begin{rmk}Fix any integer $n>1$ and let $F\in \QQ[\lam]$ be the antiderivative of $H\{n\}$ constant coefficient 0. We can compute a differential equation similar to (\ref{eqHnemo}) that $F$ satisfies and deduce properties of $F$'s roots. However, this is beyond the scope of this article. 
\end{rmk}

\begin{cor}\labelt{FsimpleNoSharedRoots}
Fix an integer $n\geq 1$ and a prime $p>\max\{2,n\}$. Let $K$ be a field of characteristic $p$. Let $F$ be the formal antiderivative of $H\{n-1\}$, both considered over $K$, with constant coefficient 0. Then $H\{n\}$ and $H\{n-1\}$, considered over $K$, share no roots \tiff $F$ has no repeated roots. In particular, $H\{n_1\},H\{n_1-1\}$ share no roots in characteristic $p$.
\end{cor}
\begin{proof}
Consider the ideal $I=(H\{n\},H\{n-1\})$ in $K[\lam]$. From \bref{pascalConnection} we have:
$$
(H\{n\},H\{n-1\}) = ((1-\lam) F' + 2n F,F') = (2n F, F') = (F,F'),
$$where the last inequality holds since $2n$ is a unit in $\FF_p$ and thus in $K$. 
Therefore, $I$ is the unit ideal \tiff $F$ is has simple roots. From \bref{diffOp}(2) we see that for $n=n_1$, indeed $F$ has no repeated roots, thus for any $p$, $H\{n_1\},H\{n_1-1\}$ share no roots in characteristic $p$. 
\end{proof}

We end this section with two useful observations for computing $FT(f)$. Let $K$ be a field. Consider a polynomial $f\in K[x_1,...,x_t]$. Denote the monomial $x_1^{\mu_1}\cdots x_t^{\mu_t}$ by $\textbf{x}^{\bm{\mu}}$ where $\bm\mu$ is the multiexponent $[\mu_1,...,\mu_t]$. Similarly, for $s$ scalars in $K$, $b_1,...,b_s$, we denote $\textbf{b}=[b_1,...,b_s]$. Now, let $\textbf{x}^{\bm{\mu}_1},...,\textbf{x}^{\bm{\mu}_s}$ be the supporting monomials of $f$. Using the usual meaning of dot product we have:
$$
f = \textbf{b} \cdot [\textbf{x}^{\bm{\mu}_1}, ..., \textbf{x}^{\bm{\mu}_s}] = b_1\textbf{x}^{\bm{\mu}_1} + ...  + b_s\textbf{x}^{\bm{\mu}_s}.
$$ 
For a multi-exponent $\textbf{k}=[k_1,...,k_t]$ we denote $\max \textbf{k}$ as the maximal power in the multiexponent $\textbf{k}$, i.e. 
$$
\max \textbf{k}=\max [k_1,...,k_t] = \max_{1\leq i \leq t} k_i.
$$
Using this notation, we have the following straightforward way to produce upper and lower bounds for $FT(f)$:
\begin{lem}\labelt{upLowBound}
Let $R=K[x_1,...,x_t]$ where $K$ is a field of prime characteristic $p$, and let $f\in R$. Let $N$ be a positive integer. Raise $f$ to the power of $N$ and collect all monomials, so that:
\begin{equation} \labelt{eqFtoNcollected}
f^N = \sum_{\text{distinct multi-exponents } \textbf{k}} c_{\textbf{k}} \textbf{x}^{\textbf{k}}.
\end{equation}
Note that all but finitely many $c_{\textbf{k}}$'s are 0. 
 Fix $e\in \ZZ_{\geq 0}$ and consider $\frac{N}{p^e}$. Then:
\begin{enumerate}
\item $ \frac{N}{p^e}<FT(f) \iff \exists \textbf{k}$ \st $c_{\textbf{k}}\neq 0$ and $\max \textbf{k} < p^e$. Or, equivalently,
\item $FT(f)\leq \frac{N}{p^e} \iff \forall \textbf{k}$, either $c_{\textbf{k}}= 0$ or $\max \textbf{k} \geq p^e$.
\end{enumerate}
\end{lem}
\begin{proof}
This is immediate from the definition (\ref{FTdef}) and from \cite[Prop 3.26]{KSbasic} which implies that for any $\frac{N}{p^e}\in [0,1]$, $$f^N \not\in (x_1^{p^e},...,x_t^{p^e})R \iff \frac{N}{p^e} < FT(f).$$
\end{proof}

\begin{lem}\labelt{betaBoundHomog} Let $f$ be a homogeneous polynomial of degree $d$ in $t$ variables. Let $\textbf{x}^{\textbf{k}}$ be a monomial in $f^N$ with a non-zero coefficient. Denote $\textbf{k}=[k_1,...,k_t]$. Then $k_1+...+k_t=dN$. Moreover, $\max \textbf{k}\geq Nd/t$ and if $\max \textbf{k} = Nd/t$ then $\textbf{k} = [Nd/t,Nd/t,...,Nd/t]$.
\end{lem}
\begin{proof}
 The first statement is immediate since any monomial of $f^N$ is of degree $dN$. Ergo, we cannot have that all $t$ entries of $\textbf{k}$ are less than $Nd/t$. Lastly, if $\max \textbf{k} = Nd/t$ but another power is less, then $k_1+...+k_t$ is less than $Nd$. 
\end{proof}

\section{Proof of The Main Theorem}
Now we are ready to prove the \nameref{mainThmv2}: 
\begin{proof}
Fix $p>2$. 
We first show that if $f_a$ is ordinary then $FT(f_a)$ is 1. 
Recall the notations: for an integer $e\geq 1$ we denote $$N_e=p^e-1$$ $$n_e = N_e/2 = (p^e-1)/2.$$ In particular,
$$
n_1 = \frac{p-1}{2}.
$$ Let us raise $f_a$ to the power of $N_e=p^e-1$. Due to \bref{hassePoly}  and \bref{betaBoundHomog} we get:
$$
f_a^N = \pm {2n_e \choose n_e}H\{n_e\}(a)x^{N_e}y^{N_e}z^{N_e} \text{mod }\mathfrak{m}^{[p^e]},
$$ where $\mathfrak{m}=(x,y,z)$ and $\mathfrak{m}^{[p^e]}=(x^{p^e},y^{p^e},z^{p^e})K[x,y,z]$.
By \bref{upLowBound}, if we show that ${2n_e \choose n_e}H\{n_e\}(a)\not \equiv 0 \pmod p$ for any $e$, then we get a lower bound of $N_e/p^e=\frac{p^e-1}{p^e}$ for $FT(f_a)$. By taking $e\to \infty$ we get that:
$$
\lim_{e\to \infty} \frac{p^e-1}{p^e} \leq FT(f_a) \leq 1 \then 1 = FT(f_a)
$$
So suffices to show that ${2n_e \choose n_e}H\{n_e\}(a)\not \equiv 0 \pmod p$.

First we deal with ${2n_e \choose n_e}$. We shall write both $2n_e$ and $n_e$ in their base $p$-expansion:
\[
\begin{array}{lrrrrrrrrrr}
2n_e = p^e-1 = &(p-1)p^1& +& (p-1)p^2& +&...&+& (p-1)p^{e-1}\\
n_e = &\frac{p-1}{2}p^1& +& \frac{p-1}{2}p^2& +&...&+& \frac{p-1}{2}p^{e-1}\\
\end{array}
\]
Since the digits of $n_e$ and $n_e$ are added without carrying to the digits of $2n_e$, by \nameref{luc} ${2n_e \choose n_e}\not \equiv 0 \pmod p$. 

Next, due to  \bref{neHasse}:
$$
 H\{n_e\}(a) = (H\{n_1\}(a))^{1+p+...+p^{e-1}}
$$
We conclude that $H\{n_e\}(a)\not \equiv 0 \pmod p$ since the polynomial is ordinary, which means that $H\{n_1\}(a)\not \equiv 0 \pmod p$.  This concludes the case where $f_a$ is ordinary.

\vskip 0.5cm

Now, we deal with the supersingular case. We shall prove that $1-1/p$ is both an upper and a lower bound for $FT(f_a)$. So fix $p>2$ and assume that $f_a$ is supersingular, i.e. that $a$ is a root of $H\{n_1\}$. We first establish $1-1/p$ as an upper bound. Let $N=p-1=2n_1$. Consider $f_a^N$ and apply \bref{hassePoly}. Because $f_a$ is supersingular, the coefficient of $x^{N}y^{N}z^{N}$ is 0 since this coefficient is a multiple of $H\{n_1\}(a)$. From \bref{betaBoundHomog}, all other monomials $\textbf{x}^\textbf{k}$ satisfy $\max \textbf{k}\geq N+1=p$. So apply \bref{upLowBound} to get an upper bound of 
$$
\frac{N}{p}=\frac{p-1}{p}=1-\frac{1}{p}
$$

As for the lower bound, fix $e\geq 1$. We will show that $\frac{p^e-p^{e-1}-1}{p^e}$ is a lower bound for all $e$, which yields a lower bound of $1-1/p$ by taking $e\to \infty$. Once we show that, the proof is complete.
We fix $e$ and set $N=p^e-p^{e-1}-1$, and we shall prove that $f^N \not \in \mathfrak{m}^{[p^e]}$. Notice that:
$$
N=p^e-p^{e-1}-1 =p^e-2p^{e-1} + p^{e-1} -1  = (p-2)(p^{e-1}) + p^{e-1} -1 = (n_1)(p^{e-1}) + (n_1-1)(p^{e-1}) + p^{e-1} -1.
$$
We set 
\[
\begin{array}{ll}
n=&(n_1)(p^{e-1})\\
m = &(n_1-1)(p^{e-1}) + p^{e-1} -1.
\end{array}
\]
 Notice that $m+1=n$.  

In order to show the lower bound, it suffices to compute the coefficient of $\textbf{x}^{2m,2n,n+m}$ in $f^N_a$ and show that it is non-zero, because:
$$
\max({2n,2m,m+n}) = 2n = (2n_1)(p^{e-1}) = (p-1)(p^{e-1})<p^e.
$$
From the \nameref{hassePolyGeneral} we get the coefficient of a critical term in $f^N$ is:
\begin{equation}\labelt{coeff}
{m+n \choose n}H\{m\}(a)
\end{equation}
We wish to prove that the coefficient (\ref{coeff}) is non-zero mod $p$. We shall break it to two parts, the binomial ${m+n \choose n}$, and the polynomials expression $H\{m\}(a)$. Let us start with the binomial coefficient. We write $m,n$ in their $p$-expansion while taking advantage of the geometric series formula:
\begin{equation}\labelt{mexpansion}
\begin{array}{lrrrrrrrrrr}
n = &(0)p^0& +& (0)p^1& +&...&+& (0)p^{e-2}&+&n_1p^{e-1}\\
m  =&(p-1)p^0& +& (p-1)p^1& +&...&+& (p-1)p^{e-2}&+&(n_1-1)p^{e-1}\\
\end{array}
\end{equation}
So when adding $m$ and $n$, the digits do not carry, as one invokes \nameref{luc} to observe that the binomial coefficient ${m+n \choose n}$ is non-zero .

We complete the proof that the coefficient (\ref{coeff}) is not zero by showing that $H\{m\}(a)$ is not zero mod $p$. Recall that by our supersingularity hypothesis $H\{n_1\}(a)\equiv 0 \pmod p$. So suffices to show that the polynomials $H\{n_1\}$ and $H\{m\}$ share no roots in characteristic $p$. Observe again the $p$-expansion of $m$ \pref{mexpansion}. Use \bref{hasseLucas} to deduce
$$
H\{m\} = H\{p-1\}^{1+p+...+p^{e-2}}H\{n_1-1\}^{p^{e-1}}
$$
So the problem is reduced to verifying that the irreducible factors of the polynomial $H\{n_1\}(\lam)\in \FF_p[\lam]$ are neither factors of $H\{p-1\}(\lam)\in \FF_p[\lam]$ nor of $H\{n_1-1\}(\lam)\in \FF_p[\lam]$. The problem does not depend on $e$. 

Let us start with $H\{p-1\}$. Recall \bref{pMinusOneHass}. Only $\lam=1$ is a root of $H\{p-1\}$ but $H\{n_1\}(1)$ is not zero due to \bref{diffOp}(1). 

It remains to compare the roots of $H\{n_1\}$ and $H\{n_1-1\}$. From \bref{FsimpleNoSharedRoots} we conclude that they share no roots, as required. This concludes the proof. 

\end{proof}

\begin{disc} For completeness, let us compute that $FT(f_a)=1/2$ for 
$$
f_a = y^2z+x(x+z)(x+a z),\, a\in K-\{0,1\}
$$ where $\mchar(K)=2$. From \bref{hasseLucas} we deduce that over $K$ and for any integer $m>0$, $H\{m\} = H\{1\}^m = (1+\lam)^m$. Since $a\neq 1$, $a$ does not satisfy any Deuring polynomial over $K$. To prove that $1/2$ is an upper bound, just observer that $f_a^1$ is already in $(x^2,y^2,z^2)$ making $1/2$ an upper bound. Now, we would like to show that $(2^{e-1}-1)/2^e$ is a lower bound for all $e$, which would result in an lower bound of $1/2$. So let 
$$N=2^{e-1}-1 = 1 + 2 + 2^2 + ... + 2^{e-3} + 2^{e-2} $$
To avoid carrying, choose $N=n+m$ with 
$$
n = 2^{e-2}, m = 2^{e-2}-1=n-1 = 1+ 2 + ... + 2^{e-3}.
$$
By construction, and due to \nameref{hassePolyGeneral}, the coefficient of $x^{2m}y^{2n}z^{n+m}$ does not vanish, while $\max\{2n,2m,m+n\}=2n = 2^{e-1}<2^e$. Thus we get an lower bound of $N/2^e = (2^{e-1}-1)/2^e$ as required.
\end{disc}

\bibliographystyle{plain}
\bibliography{unibib}

\end{document}